\documentclass[a4paper, 10pt]{article}

\usepackage[T1]{fontenc}
\usepackage[utf8]{inputenc}

\usepackage{amssymb}
\usepackage{amsmath}
\usepackage{amsfonts}
\usepackage{amsthm}
\usepackage{mathtools}
 
\usepackage{comment}
\usepackage{url}
\usepackage{float}
\usepackage[dvipsnames]{xcolor}
\usepackage{enumerate}
\usepackage{tikz}
\usepackage{caption}
\usepackage{cases}

\newcommand{\E}{\mathbb{E}}
\newcommand{\PP}{\mathbb{P}}

\newcommand{\N}{\mathbb{N}}

\DeclareMathOperator{\sgn}{sgn}

\theoremstyle{plain}%
\newtheorem{theorem}{Theorem}[section]

\newtheorem{proposition}[theorem]{Proposition}

\theoremstyle{definition}

\theoremstyle{remark}
\newtheorem{remark}[theorem]{Remark}

\usepackage{url}
\usepackage{hyperref}%
\hypersetup{
	colorlinks=false,
	linkbordercolor=gray,
	citebordercolor=gray,
	urlbordercolor=gray}%

\title{A Few  Surprising Integrals}
\author{ and Jeffrey E. Steif}

\author{Malin Palö Forsström
\thanks{Chalmers University of Technology and Gothenburg University, Gothenburg, Sweden and KTH Royal Institute of Technology, Stockholm, Sweden.\ \ Email:
        \hbox{malinpf@kth.se} }
\and         Jeffrey E. Steif
\thanks{Chalmers University of Technology and Gothenburg University, Gothenburg, Sweden.\ \ Email:
        \hbox{steif@chalmers.se}}
}

\date{\today} 

\begin{document}

\maketitle

\begin{abstract}
Using formulas for certain quantities involving stable vectors, due to\\
I. Molchanov, and in some cases utilizing the so-called  divide and color model, 
we prove that certain families of integrals
which, ostensibly, depend on a parameter are in fact independent of this parameter.

 \newline
 MSC 2010 \emph{subject classifications.}
 Primary 60E07, 60G52           
  \medskip\noindent
\end{abstract}

\section{Statement of Result and Proof}

In the pursuit of some other questions, we realized that the following integrals
  surprisingly had the same value for all  $\alpha$. 

\begin{theorem}\label{theorem:wierd integral}
The following two integrals, the first to be taken in the Cauchy principal value sense as it is not 
Lebesgue integrable, are independent of $\alpha>0$, the first having value $\pi^2/6$ and 
the second having value $\pi^2/4$.
$$
\int_0^\pi
\frac{\log \left(|\cos \theta|^\alpha + |\sin \theta |^\alpha + |\cos \theta + \sin \theta|^\alpha \right)}{\alpha \cos \theta \sin \theta} d\theta ,
$$
$$
\int_0^\pi
\frac{\log \left( \frac{|\sin \theta |^\alpha}{2} + |\cos \theta + (\frac{1}{2})^{\frac{1}{\alpha}}\sin \theta|^\alpha \right)}{\alpha \cos \theta \sin \theta} d\theta .
$$
\end{theorem}

\begin{proof}[Proof]

We will show the independence in $\alpha$ for $\alpha\in (0,2)$ and then appeal to analyticity to draw
the conclusion for all $\alpha>0$.

Let $S,S_1,S_2$ be i.i.d.\ each having a symmetric stable distribution with stability exponent 
$\alpha\in (0,2)$
and scale one; this means
that their common characteristic function is given by $f(\theta)=e^{-|\theta|^\alpha}$. 
Next, let
$$
X_1\coloneqq  \frac{S+S_1}{2^{1/\alpha}}, \,\,\, X_2\coloneqq  \frac{S+S_2}{2^{1/\alpha}} .
$$
One immediately checks (from known theory, e.g. \cite{st1994}, or by computing the characteristic functions)
that $X_1$ and $X_2$ each also has a symmetric stable distribution with stability
exponent $\alpha$ and scale one. 

(i). For the first integral, we will consider 
\begin{equation}\label{eq: cov}
\E[\sgn(X_1)\sgn(X_2)]
\end{equation}
and compute its value in two different ways.
On one hand, Corollary 6.12 in \cite{m2009} implies,
after some work, that~\eqref{eq: cov} is, for a given $\alpha\in (0,2)$, 
$$
\frac{2}{\pi^2} \int_0^\pi \frac{\log \left(|\cos \theta|^\alpha + |\sin \theta |^\alpha + |\cos \theta + \sin \theta|^\alpha \right)}{\alpha \cos \theta \sin \theta} d\theta,
$$
where the integral is to be taken in the Cauchy principal value sense.
On the other hand, one can show directly, as we do below, that
\begin{equation}\label{eq: cov1/3}
\E[\sgn(X_1)\sgn(X_2)]=1/3
\end{equation}
for each such value of $\alpha$. 
This implies that this integral is independent of $\alpha$ with value $\pi^2/6$.
In order to obtain~\eqref{eq: cov1/3}, note first that, by symmetry,
$$
\E[\sgn(X_1)\sgn(X_2)]=4\PP[\sgn(X_1)=\sgn(X_2)=1]-1
$$
and so it suffices to show that 
$$
\PP[\sgn(X_1)=\sgn(X_2)=1]=1/3.
$$
Here
\begin{align*}
\PP[\sgn(X_1)=\sgn(X_2)=1] = \E[\PP[\sgn(X_1)=\sgn(X_2)=1\mid S]]
\end{align*}
 which in turn is equal to
$$
\E[\PP[\sgn(X_1)=1\mid S]^2]=
\E[\PP[S_1\ge -S\mid S]^2]=\E[\PP[S_1\ge -S]^2].
$$
By symmetry of $S_1$, this equals
$$
\E[\PP[S_1\le S]^2]=\E[F[S]^2]
$$
where $F$ is the distribution function of $S$.
For any random variable $W$ with a continuous distribution function $G$, 
on has that $G(W)$ has a uniform distribution. It follows that this last expression is
$$
\int_0^1 x^2 dx=1/3.
$$
This completes (i).

(ii). For the second integral, we will consider 
\begin{equation}\label{eq: covagain}
\E[\sgn(X_1)\sgn(S)]
\end{equation}
and compute its value in two different ways.
On one hand, Corollary 6.12 in \cite{m2009} implies,
after some work, that~\eqref{eq: covagain} is, for a given $\alpha\in (0,2)$, 
$$
\frac{2}{\pi^2}
 \int_0^\pi \frac{\log \left( \frac{|\sin \theta |^\alpha}{2} + |\cos \theta + (\frac{1}{2})^{\frac{1}{\alpha}}\sin \theta|^\alpha \right)}{\alpha \cos \theta \sin \theta} d\theta .
$$
On the other hand, as we explain below, 
\begin{equation}\label{eq: covagain1/2}
\E[\sgn(X_1)\sgn(S)]=1/2
\end{equation}
for each such value of $\alpha$. 
This will then imply that this integral is independent of $\alpha$ with value $\pi^2/4$.
Similar to (i), symmetry yields
$$
\E[\sgn(X_1)\sgn(S)]=4\PP[\sgn(X_1)=\sgn(S)=1]-1
$$
and so it suffices to show that 
$$
\PP[\sgn(X_1)=\sgn(S)=1]=3/8.
$$
To this end, note first  that
\begin{align*}
&\PP[\sgn(X_1)=\sgn(S)=1]
=
\int_0^\infty \PP[S_1\ge -s\mid S=s]  dF(s)
\\&\qquad=
\int_0^\infty \PP[S_1\ge -s]  dF(s)=
\int_0^\infty F(s) dF(s)
\end{align*}
by symmetry of $S_1$. This becomes after the change of variables $x=F(s)$,
$$
\int_{\frac{1}{2}}^1 x dx =3/8.
$$
This completes (ii).
\end{proof}

\begin{remark}
\begin{enumerate}[(i)]
\item After having obtained the above theorem, we asked on 
Mathematics Stack Exchange if one
could more directly obtain the value of $\pi^2/6$, independent of $\alpha$, 
in the first integral. This was shown by Jack D'Aurizio, see
\url{https://tinyurl.com/y6fth8vr}.\\

\item Once we knew that $\E[\sgn(X_1)\sgn(X_2)]$ was independent of $\alpha$, of course
any formula for $\E[\sgn(X_1)\sgn(X_2)]$ would have to be independent of $\alpha$, 
in particular the formula given in Corollary 6.12 in \cite{m2009} which is the above integral.
However we would have guessed that the independence in $\alpha$ of such a formula
would have appeared in a more transparent way in the integral; surprisingly this was not the 
case. \\

\item There is an alternative argument of~\eqref{eq: covagain1/2} which we very briefly sketch.
Consider the vector $(\sgn(X_1),\sgn(S),\sgn(S_1))$. It is clear that this vector is
\(  \pm 1  \)-symmetric and has pairwise nonnegative correlations.
It follows from Proposition 2.12 in \cite{st2017} that this is then a so-called divide and color process.
This means that there is a random partition of the set $\{1,2,3\}$ so that if we first randomly partition
$\{1,2,3\}$ and then assign the same value to each element of a partition element, 
$\pm 1$ each with probability $1/2$,
independently for different partition elements, then we obtain, in distribution,
$(\sgn(X_1),\sgn(S),\sgn(S_1))$. What can this random partition look like? Since $S$ and 
$S_1$ are independent, ``$2$'' and ``$3$'' must always be put in different partition elements.
``$1$'' can never be its own partition element, since then the realization $(-1,1,1)$ would   have
positive probability. However it is clear that for $(\sgn(X_1),\sgn(S),\sgn(S_1))$,
this has zero probability. Hence the only partitions which can have positive weight are
$\{\{1,2\},\{3\}\}$ and $\{\{1,3\},\{2\}\}$ and by symmetry these must each have weight $1/2$.
It is however clear that the covariance of two variables in a divide and color process 
is simply the probabilty that they are in the same partition element, and hence 
we obtain~\eqref{eq: covagain1/2}.

\end{enumerate}
\end{remark}

One can extend the proof of the independence in $\alpha$ of the first integral to higher
dimensional integrals.
Let $S,S_1,S_2,\ldots$ be i.i.d.\ each having a symmetric stable distribution with 
stability exponent $\alpha\in (0,2)$ and scale one, and let for $i\ge  1$
$$
X_i\coloneqq  \frac{S+S_i}{2^{1/\alpha}}.
$$
We now consider \( \E [\sgn(X_1 X_2 \cdots X_n)]\), the analogue of $\E[\sgn(X_1)\sgn(X_2)]$.
By symmetry, this is zero for $n$ odd. 
The following proposition follows partially from the analysis in
Section 3.5 in \cite{st2017}. The case $n=2$ corresponds to~\eqref{eq: cov1/3}. The proof is only
sketched.

\begin{proposition}\label{lemma: st} For even $n$ and for all values of $\alpha\in (0,2)$,
\(\E [\sgn(X_1 X_2 \cdots X_n)]=1/(n+1)\)
\end{proposition}

\begin{proof}[Proof]
Clearly $(\sgn(X_1),\sgn(X_2),\ldots)$ is an infinite exchangable sequence
and hence its distribution is given, due to de Finetti's Theorem (\cite{DURRETT}), by 
\begin{equation}
\int_{s=0}^1 \Pi_s \, d\nu(s),
\end{equation}
where $\Pi_s$ denotes product measure on $\{-1,1\}^{\N}$ with density $s$
and $\nu$ is some (unique) probability measure on $[0,1]$.
It is shown in \cite{st2017} that for all $\alpha\in (0,2)$,
$\nu$ is uniform distribution on $[0,1]$. 

We now exploit a different representation of this process.
Partition the unit interval $[0,1]$ into intervals $I_1,I_2,I_3,\ldots$ where $I_i$ has length
$1/2^i$. Let $U_1,U_2,U_3,\ldots$ be i.i.d.\ uniform random variables on $[0,1]$ and 
$Z_1,Z_2,Z_3,\ldots$ be i.i.d.\ uniform random variables on $\pm 1$.
Let $V_1,V_2,V_3,\ldots$ be defined by
$$
V_i:=Z_{j(i)}
$$
where $j(i)$ is chosen so that $U_i\in I_{j(i)}$. 
For people who are familiar with Kingman's theory of 
exchangeable random partitions of the integers, 
we are just first choosing an exchangeable random partition
of the integers 
using the paintbox $(1/2,1/4,\ldots)$ (see \cite{JB06})
and then assigning the same
value $1$ or $-1$, each with probability $1/2$, to {\it all} elements in a partition element,
independently for different partition elements. 
$(V_1,V_2,V_3,\ldots)$ is clearly exchangeable and its mixing measure $\nu$ in 
de Finetti's Theorem is also uniform by Theorem 3.12 in \cite{st2017}. It follows
that for all $\alpha\in (0,2)$, $(\sgn(X_1),\sgn(X_2),\ldots)$ and
$(V_1,V_2,V_3,\ldots)$ have the same distribution. Next it is clear that
\(\E [V_1 V_2 \cdots V_n]\) is the probability that in the random 
$(1/2,1/4,\ldots)$-paintbox partition restricted to $\{1,\ldots,n\}$ there are only
partitions with an even number of elements. One can show, using induction, conditioning
on the number of terms entering the first box and using the scale
invariance of this paintbox, that
the probability of this latter event is $1/(n+1)$, completing the proof.
\end{proof}

Corollary 6.12 in \cite{m2009} provides formulas for
\(\E [\sgn(X_1 X_2 \cdots X_n)]\) in terms of integrals over the sphere 
\( \mathbb{S}^{n-1} \) which ostensibly depend on $\alpha$. However, a consequence of 
Proposition~\ref{lemma: st} now is that these higher dimensional integrals do not in fact depend on $\alpha$.

\section*{Acknowledgements}

We thank Svante Janson for pointing out that the second integral in the main theorem is in $L_1$
and so one does not need to take the integral in the Cauchy principal value sense.
The first author acknowledges support from the European Research Council, grant agreement no.\ 682537.
The second author acknowledges the support of the Swedish Research 
Council and the Knut and Alice Wallenberg Foundation.

\end{document}